\documentclass[11pt]{article}
\usepackage{lineno}
\usepackage{cite}
\usepackage{mathrsfs} 
\usepackage{ifthen}
\usepackage{adjustbox}
\usepackage{siunitx}
\usepackage{multirow}
\usepackage{exscale}
\usepackage{tabularx}
\usepackage{latexsym}
\usepackage{amsmath,amssymb,amstext,amsthm} 
\usepackage{graphicx,color,epsfig}
\usepackage{graphicx}
\usepackage{fullpage}        
\usepackage{float}
\usepackage{verbatim}
\usepackage{datetime}
\usepackage[bookmarks,pagebackref,
    pdfpagelabels=false, 
    ]{hyperref}
\usepackage{makeidx}
\usepackage{chngcntr}

\numberwithin{equation}{section}
\makeindex

\counterwithin{table}{section}

\def\R{\mathbb{R}}

\def\eqref#1{{\normalfont(\ref{#1})}}

\def\LRMC{\mbox{\bf  LRMC}\,\,}

\def\eqref#1{{\normalfont(\ref{#1})}}

\newtheorem{theorem}{Theorem}[section]

\newtheorem{proposition}[theorem]{Proposition}

\newtheorem{cor}[theorem]{Corollary}

\newtheorem{lemma}[theorem]{Lemma}

\usepackage[section]{algorithm}
\usepackage{algpseudocode}

\newcommand{\textdef}[1]{\textit{#1}\index{#1}}

\newcommand{\PP}{{\mathcal P} }

\newcommand{\A}{{\mathcal A}}

\newcommand{\bbm}{\begin{bmatrix}}
\newcommand{\ebm}{\end{bmatrix}}
\newcommand{\bem}{\begin{pmatrix}}
\newcommand{\eem}{\end{pmatrix}}
\newcommand{\beq}{\begin{linenomath*} \begin{equation}}
\newcommand{\beqs}{\begin{linenomath*} \begin{equation*}}
\newcommand{\eeq}{\end{equation} \end{linenomath*}}
\newcommand{\eeqs}{\end{equation*} \end{linenomath*}}
\newcommand{\beqr}{\begin{linenomath*} \begin{eqnarray}}
\newcommand{\beqrs}{\begin{linenomath*} \begin{eqnarray*}}
\newcommand{\eeqr}{\end{eqnarray} \end{linenomath*}}
\newcommand{\eeqrs}{\end{eqnarray*} \end{linenomath*}}
\newcommand{\bet}{\begin{table}}
\DeclareMathOperator{\Range}{Range}

\DeclareMathOperator{\Null}{Null}

\DeclareMathOperator{\rank}{{rank}}

\newcommand{\nc}{\newcommand}
\nc{\arrow}{{\rm arrow\,}}
\nc{\Arrow}{{\rm Arrow\,}}
\nc{\BoDiag}{{\rm B^0Diag\,}}
\nc{\bodiag}{{\rm b^0diag\,}}

\nc{\Mm}{{\mathcal M}^{m} }
\nc{\Mmn}{{\mathcal M}^{mn} }
\nc{\Mpq}{{\mathcal M}^{pq} }
\nc{\Mnr}{{\mathcal M}_{nr} }
\nc{\Mnmr}{{\mathcal M}_{(n-1)r} }
\nc{\kwqqp}{Q{$^2$}P\,}
\nc{\kwqqps}{Q{$^2$}Ps}

\nc{\notinaho}{(X,S)\in \overline{AHO}(\A)}
\nc{\inaho}{(X,S)\in AHO(\A)}

\newcommand{\bea}{\begin{eqnarray}}%
\newcommand{\eea}{\end{eqnarray}}%
\newcommand{\beas}{\begin{eqnarray*}}%
\newcommand{\eeas}{\end{eqnarray*}}%
\newcommand{\Rmn}{\R^{m \times n}}%
{}

\newcommand{\Hnp}[1][]{\,\mathbb{H}_+^{\ifthenelse{\equal{#1}{}}{n}{#1}}}
\newcommand{\Hn}[1][]{\,\mathbb{H}^{\ifthenelse{\equal{#1}{}}{n}{#1}}}
\newcommand{\Dn}[1][]{\,\mathbb{D}^{\ifthenelse{\equal{#1}{}}{n}{#1}}}

\begin{document}

\title{
	Uniqueness of Low Rank Matrix Completion and Schur Complement 
}

\author{
\href{https://www.researchgate.net/profile/Fei_Wang187}
{Fei Wang}\thanks{f49wang@uwaterloo.ca. Fields Institute, Toronto and Department of Combinatorics and Optimization, University of Waterloo, Waterloo, Canada }
}

\date{June, 2022}
          \maketitle

\begin{abstract}
In this paper we study the low rank matrix completion problem using tools from Schur complement. We  give a sufficient and necessary condition such that the completed matrix is globally unique with given data. We assume the observed entries of the matrix follow a special "staircase" structure. Under this assumption, the matrix completion problem is either globally unique or has infinitely many solutions (thus excluding local uniqueness). In fact, the uniqueness of the matrix completion problem totally depends on the rank of the submatrices at the corners of the "staircase". The proof of the theorems make extensive use of the Schur complement.   
\end{abstract}

{\bf Keywords:}
Low-rank matrix completion, matrix recovery, chordal graph

{\bf AMS subject classifications:}
65J22, 90C22, 65K10, 52A41, 90C46


\section{Introduction}

Given a matrix $Z$ with partially sampled data $z$, the matrix completion problem asks whether the missing entries from $z$ can be recovered. Clearly, there are infinitely many ways to fill the missing entries, so in order for the matrix completion problem to be meaningful, we require the rank of recovered matrix to be not greater than a given target rank $r$ with the assumption that the rank of the original matrix $Z$ is also not greater than $r$. 

The low rank matrix completion problem has come up naturally in many different applications, for example, in compressed sensing~\cite{candes2006}, image analysis~\cite{zhou2014}, model reduction~\cite{Liu2010}, machine learning~\cite{amit2007}, control~\cite{mesbahi1997} and so on. Numerous algorithms has been developed in recent years.  Most of the algorithms relax the non-convex rank constraint to a convex constraint and solve a convex optimization problem. 

 

\label{sect:backgr}

\subsection{Model on matrix completion}

\label{appsect:models}
We now introduce the framework for the matrix completion problem.
Suppose that we are given a low rank $m\times n$ real matrix
$Z\in \Rmn$ with $r = \rank(Z)$ where a subset of entries are \emph{sampled}.

The hard nonconvex \emph{low rank matrix completion} problem, \LRMC, for recovering the low rank matrix $Z$ can be reformulated as follows:
\begin{equation}
\label{prob:nonconvex}
(\LRMC) \qquad \qquad \begin{array}{cl}
	\mbox{find} & M  \\
	\text{s.t.} & \PP_{\hat{E}}(M) = z \\
	& \rank(M) \leq r
\end{array}
\end{equation}
\index{$\hat E$, indices of sampled entries of $Z$}
\index{indices of sampled entries of $Z$, $\hat E$}
where \emph{$\hat E$} is the set of
indices containing the known (\emph{sampled}) entries of $Z$,
$\PP_{\hat{E}}(\cdot): \Rmn \rightarrow \R^{|\hat E|}$ is the projection onto
the corresponding entries in $\hat{E}$, and
$z = \PP_{\hat{E}}(Z)$ is the vector
of known entries formed from $Z$.


\subsection{Unique recovery and related work}
Suppose we know the target rank and sampled entries $z$, a natural question to ask is which elements can be uniquely determined and which elements can not?  Also under which situation the whole low rank matrix $L$ admits a unique completion? It is known that if the linear mapping $ \PP_{\hat{E}}$ satisfies some restricted isometry properties (RIP)~\cite{candes2006,fazel2008}, then $Z$ is the unique matrices satisfying $\rank(L) \leq r,  \PP_{\hat{E}}(L) = z$. The RIP provides a sufficient condition for exact recovery and checking if a linear map $\mathcal{A}$ satisfies RIP property is NP-hard~\cite{bandeira2013}. However the necessary condition for a matrix to be uniquely completed is less well-studied. In~\cite{singer2010}, using tools from rigidity theory, a randomized algorithm for testing sufficient and necessary conditions for local completion and for testing sufficient conditions for global completion is proposed. In~\cite{kiraly2012,kiraly2015,jackson2016}, the authors use an algebraic combinatorial approach to study the matrix completion problem and a sufficient and necessary criterion is given in~\cite{kiraly2012}. However, the matrix to be completed is assumed to be ``generic" in~\cite{kiraly2012,kiraly2015,jackson2016}. 

We are mostly interested in exploiting the clique structures in the sampled data. The cliques in the sampled data can be used to do facial reductions for solving the corresponding optimization problems more efficiently~\cite{ma2020,huang2018}. 
In the following sections, we will provide a sufficient and necessary condition for the uniqueness of completion without using the generic assumption by exploiting the properties of the intersections of bicliques. 


\section{Matrix completion by clique intersection}
We first recall the basic property of Schur complement:
\begin{lemma}\label{lem:schur}\cite{hornzhang2005}
Consider the partitioned matrix
\[
M = \begin{bmatrix}
 A & B \\ 
 C & D
\end{bmatrix}\in \R^{m\times n},
\]
if we assume that  $\Range(B) \subseteq \Range(A)$ and $\Range(C^T) \subseteq \Range(A^T)$, then $M/A = D - C A^{\dagger} B$ is well-defined and 
\[
\begin{bmatrix} 
I & 0 \\
-CA^{\dagger} & I
\end{bmatrix} 
\begin{bmatrix}
A & B \\
C & D 
\end{bmatrix}
\begin{bmatrix}
I & -A^{\dagger} B \\
0 & I
\end{bmatrix} = 
\begin{bmatrix}
A & 0 \\
0 & M/A
\end{bmatrix}
\] and hence 
$$\rank(M) = \rank(A) + \rank(M/A).$$
\end{lemma}
\subsection{Two-biclique intersection}
As a consequence of Lemma~\ref{lem:schur}, we have the following theorem about the completion of $2 \times 2$ block matrix with one missing block.
\begin{theorem} \label{thm:uniq}
Consider the partitioned matrix
\[
M = \begin{bmatrix}
 A & B \\ 
 C & D
\end{bmatrix}\in \R^{m\times n},
\]
with $\rank(M) = r$, and the block submatrices $A,B,C$ are fixed (sampled).
Then $M$ is unique in LRMC if and only if $\rank(A) = r$.
\end{theorem}
\begin{proof}
First assume $\rank(A) = r$. Let $M/A = D-C A^{\dagger} B$ be the
generalized Schur complement, e.g.,~\cite{ouel:81}, 
where $A^\dagger$ denotes the
Moore-Penrose generalized inverse. Then we 
have $\rank(M/A) + \rank(A) =\rank(M)$. Therefore $\rank(M/A) = 0$ and $D = CA^{\dagger}B$ is unique.
\index{$A^\dagger$, Moore-Penrose generalized inverse}
\index{Moore-Penrose generalized inverse, $A^\dagger$}
\index{$M/A$, generalized Schur complement}
\index{generalized Schur complement, $M/A$}

For necessity, assume $\rank(A) <r$, we first assume $\Range(B) \subseteq \Range(A)$ and $\Range(C^T) \subseteq \Range(A^T)$. Then we by Lemma \ref{lem:schur} we have the following equality 
\[
\rank(D-CA^{\dagger}B) + \rank(A) = \rank(M)
\]
Since $\rank(A) < r$ and $\rank(M) = r$, we have $\rank(D - C A^{\dagger}B) = \rank(M) -\rank(A) = \bar r >0$. We can then let $D = C A^{\dagger}B + E$ where $E$ is an arbitrary matrix of rank $\bar r$. Therefore $D$ is not unique.

Now suppose either $\Range(B) \subsetneq \Range(A)$ or $\Range(C^T) \subsetneq \Range(A^T)$. Without loss we can assume $\Range(C^T) \subsetneq \Range(A^T)$. Then we have $\Null(A) \subsetneq \Null(C)$ which means there exists a column vector $x$ such that $Ax =0, Cx \neq 0$. Now we can add the column vector 
$\begin{bmatrix} A x \cr C x \end{bmatrix} = \begin{bmatrix} 0 \cr C x \end{bmatrix} $ to any column of $\begin{pmatrix} B\cr D
\end{pmatrix}$ without changing the rank of $M$ and we get a different $D$. So $D$ is not unique.
\end{proof}
We now study the simple case such that the observed data pattern $\hat{E}$ consists of two bicliques . 
\begin{cor}\label{cor:2cliq}
Let $Z$ be the following matrix with two intersecting
bicliques  such that $\rank(Z)=r$ and  submatrices $X$ and $Y$ are fixed, 
\begin{equation}
Z=\left[\begin{array}{ccc}
 \multicolumn{1}{c|}{Z_1} &X_1 & X_2 \\ \cline{1-2}
\multicolumn{1}{c|}{Y_1}  & \multicolumn{1}{c|}{Q}  &X_3  \\ \cline{2-3}
Y_2 & \multicolumn{1}{c|}{Y_3}  & Z_2
\end{array}\right],
 \quad X=\left[\begin{array}{cc}
 X_1 & X_2 \\ \cline{1-1}
 \multicolumn{1}{c|}{Q}  &X_3  
\end{array}\right],
\quad Y =\left[\begin{array}{cc}
\multicolumn{1}{c|}{Y_1}  & Q    \\ \cline{2-2}
Y_2 & Y_3  
\end{array}\right].
\end{equation}
Let $Q$  be the submatrix that lies in both $X$ and $Y$.  
If $\rank(Q) = r$, then $Z$ is unique in LRMC.
\end{cor}
\begin{proof}
If $Q$ has rank $r$, then the top left union of four blocks must also be
rank $r$. Therefore, Theorem \ref{thm:uniq} implies $Z_1$ is unique.
Similarly, $Z_2$ is unique by looking at the four blocks at the bottom
right.
\end{proof}
However, the converse may not be true. Consider the following example 
\begin{equation}
Z=\left[\begin{array}{cccc}
 \multicolumn{1}{c|}{Z_1} & 6 & 5& 3 \\ \cline{1-3}
\multicolumn{1}{c|}{1}  & 2 &\multicolumn{1}{c|}{3}  &2  \\ \cline{2-4}
3 & 4 &\multicolumn{1}{c|}{2}  & Z_2
\end{array}\right],
 \quad Q=\left[\begin{array}{cc}
 2 & 3  
\end{array}\right],
\end{equation}
Assume $\rank(Z) = 2$ and $\rank(Q) = 1 < \rank(Z)$. However, $Z_1$ and $Z_2$ are still unique and by basic linear algbera we have $Z_1 = 4$ and $Z_2 = 1$.

From Theorem \ref{thm:uniq}, if we have two bicliques such that their 
intersection has the target rank, we can now merge these two bicliques 
into one bigger biclique and recover the corresponding missing entries of $Z$.
We can then use this bigger biclique to merge with other bicliques. This
process can carry on until all the missing entries are recovered.

\subsection{More bicliques and the staircase }
We can generate Theorem~\ref{thm:uniq} to include three  bicliques.
\begin{theorem} \label{thm:uniq2}
Consider the partitioned matrix
\[
M = \begin{bmatrix}
E & F \\ 
A & B \\ 
 C & D
\end{bmatrix}\in \R^{m\times n},
\]
with $\rank(M) = r$, where $A,B,C,F$ are fixed (sampled).
Then $M$ is unique in LRMC if and only if $\rank(A) = r$ and $\rank(B) = r$.
\end{theorem}
\begin{proof}
Suppose $\rank(A) = \rank(B) = r$, then it is obvious that $D,E$ are unique by Theorem \ref{thm:uniq}.

For necessity, without loss we assume $\rank(A) = \bar r < r$. By a permutation, let
\[
M = \begin{bmatrix}
A & B \\ 
E & F \\ 
C & D
\end{bmatrix}.
\]

If $ \rank(\begin{bmatrix}
A  \\ 
E 
\end{bmatrix}) < r$, then by Therem \ref{thm:uniq}, $D$ is not unique so $M$ is not unique.
 
If $\rank(\begin{bmatrix}
A  \\ 
E  
\end{bmatrix}) = r$, then we have $\Range( B )\subseteq \Range(A )$. Now we partition $A,E,C$ such that 
\[
M = \begin{bmatrix}
A_1 & A_2 & B \\ 
E_1 & E_2 & F \\ 
C_1 & C_2 & D
\end{bmatrix}
\]
where $A_1$ has full column rank $\bar r$. So we have $\Range(A_1) = \Range(A)$.

Let $M_1, M_2, M_3, M_4$ be the four Schur complements corresponding to $E_2, F, C_2, D$ such that $M_1 = E_2 - E_1 A_1^{\dagger} A_2$, $M_2 = F - E_1 A_1^{\dagger} B$,
$M_3 = C_2 - C_1 A_1^{\dagger} A_2$, $M_4 = D - C_1 A_1^{\dagger} B$. Since $\Range(A_2) \subseteq \Range(A_1)$, $\Range(B) \subseteq \Range(A_1)$ and $\Range(E_1^T) \subseteq \Range(A_1^T)$, $\Range(C_1^T) \subseteq \Range(A_1^T)$, we have 
\[
\rank(M) = \rank(A_1) + \rank(\begin{bmatrix} M_1 & M_2 \cr M_3 & M_4\end{bmatrix}) = r.
\]
Also 
\[
\rank(\begin{bmatrix} A \cr E\end{bmatrix}) = \rank(A_1) + \rank(M_1) = r.
\]  
Therefore $\rank(M_1) =  \rank(\begin{bmatrix} M_1 & M_2 \cr M_3 & M_4\end{bmatrix}) = r - \bar r$ and we have 
\begin{equation}\label{eq:m14}
M_4 = M_3 M_1^{\dagger}M_2.
\end{equation}
Now $M_1 \neq 0$, since $\rank(M_1) = r - \bar r >0$, we can perturb $E_2$ such that 
$\bar E_2 = E_2 + M_1$ and perturb $D$ such that $\bar D = D - \frac12 M_4$ and the corresponding full perturbed matrix is $\bar M$. After similar arguments we can get 
\begin{eqnarray*}
\rank(\bar M) & = & \rank(A_1) + \rank(\begin{bmatrix} 2 M_1 & M_2 \cr M_3 & \frac12 M_4\end{bmatrix}) \\
& = & \rank(A_1) + \rank(2 M_1) + \rank(\frac12 M_4 - M_3(2M_1)^{\dagger}M_2) \\
& = & \rank(A_1) + \rank(2 M_1)  \quad \text{ (due to \eqref{eq:m14})} \\
& = & \bar r + r - \bar r = r.
\end{eqnarray*}
Therefore $M$ is not unique.
\end{proof}
 The following proposition regarding the general case of 3 bicliques is a direct consequence of Theorem~\ref{thm:uniq2}.     
\begin{proposition} \label{prop:uniq3}
Consider the partitioned matrix
\[
M = \begin{bmatrix}
F & H  & E \\ 
A & G  & B \\ 
C & K  & D 
\end{bmatrix}\in \R^{m\times n},
\]
with $\rank(M) = r$, where $A,B,C,E,G$ are fixed (sampled).
Then $M$ is unique in LRMC if and only if $\rank(A) = r$ and $\rank(B) = r$.
\end{proposition}
\begin{proof}
Without loss we assume $\rank(A)< r$. Now let $B = [G,\;B]$, $E = [H,\;E]$, the result follows directly from Theorem~\ref{thm:uniq2}.
%
\end{proof}

We can then generate Theorem~\ref{thm:uniq2} to the cases with 4 and 5 bicliques in the following Theorem and propositions.
\begin{theorem} \label{thm:uniq4}
Consider the partitioned matrix
\[
M = \begin{bmatrix}
F & H  & E \\ 
A & G  & B \\ 
C & K  & D 
\end{bmatrix}\in \R^{m\times n},
\]
with $\rank(M) = r$, where $F,A,G,K,D$ are fixed.
Then $M$ is unique in LRMC if and only if $\rank(A) = \rank(G) = \rank(K) = r$.
\end{theorem}
\begin{proof}
If $\rank(A)< r$ or $\rank(K) <r$, then according to Theorem \ref{thm:uniq2}, $M$ is not unique. Therefore we only need to consider the case when $\rank(G) < r$.

If $\rank(\begin{bmatrix} H \\ G\end{bmatrix}) < r$ or $\rank(\begin{bmatrix} G & B\end{bmatrix}) < r$, then again according to Theorem \ref{thm:uniq2}, $M$ is not unique. Therefore we consider the case where 
$\rank(\begin{bmatrix} H \\ G\end{bmatrix}) = r$ and $\rank(\begin{bmatrix} G & B\end{bmatrix}) = r$.

By a permutation, let 
\[
M = \begin{bmatrix}
G & B  & A \\ 
H & E  & F \\ 
K & D  & C 
\end{bmatrix}\in \R^{m\times n}.
\]
Let $P = \begin{bmatrix} G & B \\ H & E \end{bmatrix}$, since $\rank(G) < r$, by Theorem \ref{thm:uniq}, there exists a different $\bar{E}$  and $\bar{P}$ such that $\rank(\bar{P}) = r$, we let $\bar{C} = \begin{bmatrix} K & D \end{bmatrix}\bar{P}^{\dagger} \begin{bmatrix} A  \\ F \end{bmatrix})$,  since  $\Range(\begin{bmatrix}A \\ F \end{bmatrix}) \subseteq  \Range(\begin{bmatrix}G \\ H \end{bmatrix})$ and  $\Range(\begin{bmatrix}K^T \\ D^T \end{bmatrix}) \subseteq  \Range(\begin{bmatrix}G^T \\ B^T \end{bmatrix})$, we have $\rank(\bar{M}) = \rank(\bar{P}) + \rank(\bar{C} - \begin{bmatrix} K & D \end{bmatrix} \bar{P}^{\dagger} \begin{bmatrix} A  \\ F \end{bmatrix}) = \rank(\bar{P}) = r$. The corresponding $\bar{M}$ is different from $M$ and the proof is finished.
\end{proof}

\begin{proposition} \label{prop:uniq5}
Consider the partitioned matrix
\[
M = \begin{bmatrix}
F & H  & E \\ 
A & G  & B \\ 
C & K  & D  \\
J & I  & L
\end{bmatrix}\in \R^{m\times n},
\]
with $\rank(M) = r$, and $F,A,G,K,D, L$ are fixed.
Then the matrix $M$ is unique in LRMC if and only if $\rank(A) = \rank(G) = \rank(K) = \rank(D) = r$.
\end{proposition}
\begin{proof}
Direct consequences from Theorem \ref{thm:uniq2} and Theorem \ref{thm:uniq4}.
\end{proof}

With all the prerequisite theorems and propositions that have been proved, we are now able to extend to the most general case where we have a stair case of known block matrices. We conclude that the whole matrix is unique if and only if every ``corner" matrix has rank $r$. This can be proved by repeatedly using the previous theorems and propositions. 
\begin{theorem}\label{thm:mtx}
         Given a low rank matrix $Z \in \R^{m\times n}$ and a partial sampling $\PP_{\hat{E}}(Z) = z$. If by a permutation there exists a chain of bicliques (submatrices of $Z$) $\alpha_1, ..., \alpha_l$  with the corresponding edge sets $E_1,\cdots, E_l$ and vertex sets in pairs $(U_1,V_1),(U_2,V_2),\cdots,(U_l,V_l)$  . Assume $\cup_{i=1}^l E_i  = \hat{E}$ and $E_i \cap E_{j} = \emptyset, \, \forall j \geq i+2$ and the union of all the vertices of the bicliques satisfy $\cup_{i=1}^l \alpha_i = \{1,...,m\} \times \{1,...,n \}$. In addition
         \begin{itemize}
             \item   $U_i \cap U_{i+2} = U_{i+1}$, $V_{i}\cap V_{i+1} = V_i, V_{i+2} \cap V_{i+1} = V_{i+2}$ for $i $ being even  number between $1$ and $l-2$ 
             \item  $V_i \cap V_{i+2} = V_{i+1}$, $U_{i}\cap U_{i+1} = U_i, U_{i+2} \cap U_{i+1} = U_{i+2}$ for $i $ being odd  number between $1$ and $l-2$ 
         \end{itemize}
         Then the matrix $Z$ can be uniquely recovered if and only if $\rank (X_{\alpha_i \cap \alpha_{i+1}}) = r, i = 1,...,l-1$.     
\end{theorem}
\begin{proof}

Suppose the bicliques satisfy $\rank (X_{\alpha_i \cap \alpha_{i+1}}) = r, i = 1,...,l-1$, then it is abvious to see the uniqueness of the completion by repeatedly using Theorem \ref{thm:uniq}.

For the other direction, suppose one of the bicluqes $X_{\alpha_i}$ satisfies $\rank(X_{\alpha_i}) \neq r$, then one can recover a different matrix $\bar{Z}$ with the same rank and sampled data by reducing the problem to the sample cases as shown in Theorem \ref{thm:uniq}, \ref{thm:uniq2}, \ref{thm:uniq4} and Proposition~\ref{prop:uniq3}, \ref{prop:uniq5}.
\end{proof}

An illustration of the ``staircase" structure in the assumption of Theorem~\ref{thm:mtx} is shown in equation \eqref{eq:staircase}.

\begin{equation}\label{eq:staircase}
Z=\left(\begin{array}{cccccccc}
 \multicolumn{1}{|c|}{} &  &  &  & & & &\\ \cline{2-3}
\multicolumn{1}{|c}{Z_1}  &  & \multicolumn{1}{c|}{Z_2} &  & & & &\\ \cline{1-2}
 &   & \multicolumn{1}{|c|}{} &  & & & &\\ \cline{4-4}
& & \multicolumn{1}{|c}{Z_3} & & \ddots & & &\\
\cline{3-4} 
& & & & & \ddots & &\\
& & & & & \multicolumn{1}{|c|}{} & &\\\cline{7-8}
& & & & & \multicolumn{1}{|c}{Z_{l-1}}  & &\\
\cline{6-8}
\end{array}\right), 
\end{equation}
\section{Graph Representation of the Problem}
\label{sect:graphs}
Our sampling yields elements $b=\PP_{\hat E} (Z)$. With the matrix $Z$
and the sampled elements we can associate a bipartite graph
\textdef{$G_Z=(U_m,V_n,\hat E)$}, where
\[
U_m=\{1,\ldots,m\},\quad V_n=\{1,\ldots,n\}.
\]
For our needs we associate $Z$ with the
\textdef{undirected graph, $G=(V,E)$},
\index{$G=(V,E)$, undirected graph}
with node set
$V=\{1,\ldots,m,m+1,\ldots,m+n\}$ and edge set $E$ that satisfies
\[
	\big\{ \{ij\in V\times V: i< j\leq m\}
	\cup \{ij\in V\times V: m+1\leq i < j\leq m+n\}\big\}
\subseteq E\subseteq \{ij\in V\times V: i< j\}.
\]
Note that as above,
$\bar{E}$ is the set of edges excluding the trivial ones, that is,
\[
	\bar{E} = E \backslash \bigg\{ \{ij\in V\times V: i\leq j\leq m\}
\cup \{ij\in V\times V: m+1\leq i\leq j\leq m+n\}\bigg\}.
\]

Recall that a \textdef{biclique} $\alpha$ in the graph $G_Z$ is a complete
bipartite subgraph in $G_Z$ with corresponding complete submatrix $z[\alpha]$.\index{complete submatrix, $z[\alpha]$}
\index{$z[\alpha]$, complete submatrix}This corresponds to a nontrivial\footnote{For $G$ we have the additional
trivial cliques of size $k$, $C=\{i_1,\ldots, i_k\}\subset \{1,\ldots, m\}$ and
$C=\{j_1,\ldots, j_k\}\subset \{m+1,\ldots, m+n\}$, that are not of
interest to our algorithm.}
\textdef{clique} in the graph $G$, a complete subgraph in $G$.
The cliques of interest are $C=\{i_1,\ldots, i_k\}$ with cardinalities
\begin{equation}
	\label{eq:cardspq}
	|C\cap \{1,\ldots, m\}|=p \neq 0, \quad
	|C\cap \{m+1,\ldots, m+n\}|=q \neq 0.
\end{equation}
The submatrix $z[\alpha]$ of $Z$ for the corresponding biclique from the
clique $C$ is
\begin{equation}
	\label{eq:Xspecif}
z[\alpha]\equiv X\equiv \{Z_{i(j-m)}: ij \in C\}, \quad \text{sampled
	$p\times q$ rectangular submatrix}.
\end{equation}
These non-trivial cliques in $G$ that correspond to bicliques
of $G_Z$ are at the center of our interest.

\subsection{Chordal graphs and clique trees}
Chordal graph is a special graph with a chordless structure. 
An undirected graph is chordal if every
cycle of length greater than three has a chord.

A clique tree of a graph $G = (V,E)$ is a tree which has the cliques
of $G$ as its vertices. A clique tree $T$ has the induced subtree property
if for every $v \in V$, the cliques that contain $v$ form a subtree (connected
subgraph) of $T$. It has been shown
that chordal graphs are exactly the graphs for which a clique tree with
the induced subtree property exists. (\cite[Theorem 2.7]{buneman:1974} and~\cite[Theorem 3]{gavril1974}).

 We show the graph $\hat{E}$ in Theorem \ref{thm:mtx} has a unique clique tree with the induced subtree property. Therefore it is chordal. 

\begin{theorem}\label{thm:graph}
Given  a graph $G$ with edge set $\hat{E}$, if by a permutation there exists a chain of bicliques $\alpha_1, ..., \alpha_l$ with the corresponding edge sets $E_1,\cdots, E_l$. Assume $E_1,\cdots, E_l$ and $\alpha_1, ..., \alpha_l$ satisfy the assumptions in Theorem~\ref{thm:mtx}.
Then the graph $\hat{E}$ is chordal and it has a unique clique tree with the induced subtree property.
\end{theorem}
\begin{proof}
Consider the set of bicliques $\alpha_1, \cdots, \alpha_l$, each clique $\alpha_i$ is a node and connect $\alpha_i$ and $\alpha_{i+1}$ by an edge, then it forms a clique tree. For any vertex $v$ in $\hat{E}$, it either belongs to $\alpha_i$ or belongs to $\alpha_i \cap \alpha_i+1$ for some $i$, therefore the set of cliques containing $v$ is connected. Hence it has the induced subtree property. Therefore by Theorem 3 in \cite{gavril1974} it is a chordal graph. Now consider a different clique tree, then there exists $\alpha_i$ which is not connected $\alpha_{i+1}$, but there exists $v \in \alpha_i \cap \alpha_i+1$, therefore the induced subtree property is not satisfied. 
\end{proof}

\begin{proposition}
Given a graph $G$ with $\hat{E}$ which satisfies the chain clique properties defined in Theorem~\ref{thm:graph}, then there exists a polynomial algorithm to determine if the correponding low rank matrix completion problem \ref{prob:nonconvex} has a unique solution. 
\end{proposition}
\begin{proof}

Given a chordal graph, finding the clique tree with the induced subtree property can be solved in polynomial time by the maximum cardinality search algorithm (MCS) ~\cite{tarjan1983,blair1993}. 

After we the clique tree is found, checking if the intersection of two neighbouring clique has rank $r$ can also be done in polynomial time.
\end{proof}




\section{Completion of Positive Semidefinite Matrices is not true}

We recall the following theorem about symmetric matrix:
\begin{theorem}\label{thm:psdschur}
Suppose M is symmetric and partitioned as 
\[
M = \begin{bmatrix}
A & B \\
B^* & C
\end{bmatrix},
\]
in which $A$ and $C$ are square. Then $M \succeq 0$ if and only if $A \succeq 0, \Range(B) \subseteq \Range(A),$ and $M / A \succeq 0$.
\end{theorem}
For positive semidefinite matrix completion, we have the following proposition:
\begin{proposition} \label{thm:psduniq}
Consider the partitioned matrix
\[
M = \begin{bmatrix}
 A & B \\ 
 B^T & C
\end{bmatrix}\in \R^{m\times n},
\]
where $\rank(M) = r$ and $M \succeq 0$ in which $A$, $C$ are square, and the blocks $A,B$ are fixed.
Then there exists a unique positive semidefinite matrix $M$  if and only if $\rank(A) = r$.
\end{proposition}
\begin{proof}
First assume $\rank(A) = r$. Let $M/A = C-B^T A^{\dagger} B$ be the
generalized Schur complement, 
where $A^\dagger$ denotes the
Moore-Penrose generalized inverse. The existence of $M \succeq 0$ ensures $\Range(B) \subseteq \Range(A)$ and $A \succeq 0$, therefore $M/A$ is well-defined. We then 
have $\rank(M/A) + \rank(A) = \rank(M)$. Hence $\rank(M/A) = 0$ and $C = B^T A^{\dagger}B \succeq 0$ is unique. 

For necessity, assume $\rank(A) <r$, the existence of $M \succeq 0$ ensures $\Range(B) \subseteq \Range(A)$ and $A \succeq 0$, therefore $M/A$ is well-defined and
\[
\rank(M/A) + \rank(A) = \rank(M)
\]
Since $\rank(A) < r$ and $\rank(M) = r$, we have $\rank(C - B^T A^{\dagger}B) = \rank(M) -\rank(A) = \bar r >0$. We can then let $C = B^T A^{\dagger}B + E$ where $E$ is an arbitrary positive semidefinite matrix of rank $\bar r$. 
Hence by Theorem \ref{thm:psdschur} 
\[
\bar{M} = \begin{bmatrix}
A & B \\
B^T & B^TA^{\dagger} B  + E
\end{bmatrix} \succeq 0.
\]
Therefore $M$ is not unique.
\end{proof}

Now ee consider the following question:
let $M$ be the partitioned matrix
\[
M = \begin{bmatrix}
 A & B & D\\ 
 B^T & C & E \\
D^T & E^T & F
\end{bmatrix}\in \R^{m\times n},
\]
$M$ is symmetric and positive semidefinite and the diagonal elements of $M$ are all nonzeros. Suppose $A, B, C, E, F$ are fixed and $\rank(M) = r$. Then it is easy to see that $M$ can be uniquely completed if $\rank(C) = r$. We now ask the question: is the converse also true? The answer is NO.

We consider the following example

\[
H = \left[
\begin{array}{ccc}
    5 & 4 & -2 \\
    4 & 16 & -8 \\
    -2 & -8 & 4 
\end{array}
\right]
\]
Suppose the top right element and bottom left element $-2$ are missing, the partially observed entries are the union of two 2 by 2 cliques, i.e,
\[
\left[
\begin{array}{ccc}
    5 & 4 & ? \\
    4 & 16 & -8 \\
    ? & -8 & 4 
\end{array}
\right]
\]

We partition $H$ as matrix $M$. 
The intersection of the two cliques is 16 which is rank 1, and $\rank(H) = 2$. Let the missing element $D$ be $x$.

We compute  $\rank(\begin{bmatrix}
5 & 4 \\
4 & 16
\end{bmatrix}) = 2$. In order for the rank of $H$ to be 2, the Schur complement 
\[
4 - (x, -8)\begin{bmatrix}
5 & 4 \\
4 & 16
\end{bmatrix}^{-1}\begin{bmatrix}
x \\
-8
\end{bmatrix} = -(\frac{x^2}{4}+x + 1) = - (\frac{1}{2}x + 1)^2
\]
must be zero. (otherwise the rank of $H$ is 3.)

Therefore $x = -2$ is the only solution, the missing entry can be uniquely completed.

\section{Conclusion}
In this paper, we show that if the missing entries of a matrix of rank $r$ has a "staircase" pattern, then the uniqueness of the recovery problem given the target rank depends totally on the submatrices at the corner of the "staircase". The proof of the main theorem makes extensive use of Schur complement.  

\section*{Acknowledgment}
The author thanks Henry Wolkowicz and Anders Forsgren for the very helpful discussion during this work.
\addcontentsline{toc}{section}{Acknowledgment}

\bibliographystyle{plain}
\bibliography{references}

\begin{thebibliography}{10}

\bibitem{amit2007}
Yonatan Amit, Michael Fink, Nathan Srebro, and Shimon Ullman.
\newblock Uncovering shared structures in multiclass classification.
\newblock In {\em Proceedings of the 24th international conference on Machine
  learning}, pages 17--24, 2007.

\bibitem{bandeira2013}
Afonso~S Bandeira, Edgar Dobriban, Dustin~G Mixon, and William~F Sawin.
\newblock Certifying the restricted isometry property is hard.
\newblock {\em IEEE transactions on information theory}, 59(6):3448--3450,
  2013.

\bibitem{blair1993}
Jean~RS Blair and Barry Peyton.
\newblock An introduction to chordal graphs and clique trees.
\newblock In {\em Graph theory and sparse matrix computation}, pages 1--29.
  Springer, 1993.

\bibitem{buneman:1974}
Peter Buneman.
\newblock A characterisation of rigid circuit graphs.
\newblock {\em Discret. Math.}, 9(3):205--212, 1974.

\bibitem{candes2006}
Emmanuel~J Candes, Justin~K Romberg, and Terence Tao.
\newblock Stable signal recovery from incomplete and inaccurate measurements.
\newblock {\em Communications on Pure and Applied Mathematics: A Journal Issued
  by the Courant Institute of Mathematical Sciences}, 59(8):1207--1223, 2006.

\bibitem{fazel2008}
Maryam Fazel, E~Candes, Benjamin Recht, and P~Parrilo.
\newblock Compressed sensing and robust recovery of low rank matrices.
\newblock In {\em 2008 42nd Asilomar Conference on Signals, Systems and
  Computers}, pages 1043--1047. IEEE, 2008.

\bibitem{gavril1974}
Fǎnicǎ Gavril.
\newblock The intersection graphs of subtrees in trees are exactly the chordal
  graphs.
\newblock {\em Journal of Combinatorial Theory, Series B}, 16(1):47--56, 1974.

\bibitem{hornzhang2005}
Roger~A Horn and Fuzhen Zhang.
\newblock Basic properties of the schur complement.
\newblock In {\em The Schur Complement and Its Applications}, pages 17--46.
  Springer, 2005.

\bibitem{huang2018}
Shimeng Huang and Henry Wolkowicz.
\newblock Low-rank matrix completion using nuclear norm minimization and facial
  reduction.
\newblock {\em Journal of Global Optimization}, 72(1):5--26, 2018.

\bibitem{jackson2016}
Bill Jackson, Tibor Jord{\'a}n, and Shin-ichi Tanigawa.
\newblock Unique low rank completability of partially filled matrices.
\newblock {\em Journal of Combinatorial Theory, Series B}, 121:432--462, 2016.

\bibitem{kiraly2012}
Franz Kir{\'a}ly and Ryota Tomioka.
\newblock A combinatorial algebraic approach for the identifiability of
  low-rank matrix completion.
\newblock {\em arXiv preprint arXiv:1206.6470}, 2012.

\bibitem{kiraly2015}
Franz~J Kir{\'a}ly, Louis Theran, and Ryota Tomioka.
\newblock The algebraic combinatorial approach for low-rank matrix completion.
\newblock {\em J. Mach. Learn. Res.}, 16(1):1391--1436, 2015.

\bibitem{Liu2010}
Zhang Liu and Lieven Vandenberghe.
\newblock Interior-point method for nuclear norm approximation with application
  to system identification.
\newblock {\em SIAM Journal on Matrix Analysis and Applications},
  31(3):1235--1256, 2010.

\bibitem{ma2020}
Shiqian Ma, Fei Wang, Linchuan Wei, and Henry Wolkowicz.
\newblock Robust principal component analysis using facial reduction.
\newblock {\em Optimization and Engineering}, 21(3):1195--1219, 2020.

\bibitem{mesbahi1997}
Mehran Mesbahi and George~P Papavassilopoulos.
\newblock On the rank minimization problem over a positive semidefinite linear
  matrix inequality.
\newblock {\em IEEE Transactions on Automatic Control}, 42(2):239--243, 1997.

\bibitem{ouel:81}
Diane~Valerie Ouellette.
\newblock Schur complements and statistics.
\newblock {\em Linear Algebra and its Applications}, 36:187--295, 1981.

\bibitem{singer2010}
Amit Singer and Mihai Cucuringu.
\newblock Uniqueness of low-rank matrix completion by rigidity theory.
\newblock {\em SIAM Journal on Matrix Analysis and Applications},
  31(4):1621--1641, 2010.

\bibitem{tarjan1983}
Robert~Endre Tarjan.
\newblock {\em Data structures and network algorithms}.
\newblock SIAM, 1983.

\bibitem{zhou2014}
Xiaowei Zhou, Can Yang, Hongyu Zhao, and Weichuan Yu.
\newblock Low-rank modeling and its applications in image analysis.
\newblock {\em ACM Computing Surveys (CSUR)}, 47(2):1--33, 2014.

\end{thebibliography}

\end{document}